\documentclass[12pt]{amsart}
\pdfoutput=1
\usepackage{amssymb, amsmath, amsthm}
\usepackage{graphicx} 
\usepackage{textcomp}
\usepackage{blkarray}
\usepackage{easybmat}
\usepackage{multirow}
\usepackage{graphicx}

\newtheorem{thm}{Theorem}
\newtheorem*{main}{Theorem \ref{main}}
\newtheorem*{thm_surgery}{Theorem \ref{thm_surgery}}
\newtheorem*{thm_alt}{Theorem \ref{thm_alt}}
\newtheorem{lem}[thm]{Lemma}

\theoremstyle{definition}		

\newtheorem{definition}{Definition}
\newtheorem{example}{Example}

\newtheorem{remark}{Remark}

\newtheorem{step}{Step}
\newtheorem{stepp}{Step}

\title{On the fibration of augmented link complements}
\author{Darlan Gir\~ao}

\begin{document}
\maketitle

\begin{abstract}
We study the fibration of augmented link complements. Given the diagram of an augmented link we associate a spanning surface and a graph. We then show that this surface is a fiber for the link complement if and only if the associated graph is a tree.  We further show that fibration is preserved under Dehn filling on certain components of these links. This last result is then used to prove that  within a very large class of links, called locally alternating augmented links, every link is fibered.

\end{abstract}


\section{Introduction}
\label{intro}
Let $K$ be an oriented link in $S^3$.  By a  Seifert surface $S$  we mean an orientable spanning surface for $K$, i.e., $\partial S=K$ and the orientation of $S$ agrees with that of $K$. It is known that every oriented link has a Seifert surface (see for instance \cite{Ro}).  We say  the link $K$ is \textit{fibered} if $S^3-K$   has the structure of a surface bundle over the circle, i.e., if  there exists a Seifert surface $S$ such that $S^3-K\cong (S\times[0,1])/\phi$, where $\phi$ is a homeomorphism of $S$. In  this case we abuse terminology and say \textit{$S$ is a fiber for $K$}. 

 The study of the  fibration of link complements has been an active line of research in low dimensional topology. 
 In \cite{Ha} Harer showed how to construct all fibered knots and links. However,  deciding whether or not a link $K$  is fibered is in general a very hard problem.   Stallings \cite{St} proved that a link $K$ is fibered if and only if $\pi_1(S^3-K)$ contains a finitely generated normal subgroup whose quotient is $\mathbb{Z}$. Stallings' result is very general, but  hard to verify, even if we restrict to particular families of links.  In the  early 60's Murasugi \cite{Mu} proved that an alternating link is fibered if and only if its reduced Alexander polynomial is monic.   In \cite{Ga} Gabai proved that if a Seifert surface $S$ can be decomposed as the \textit{Murasugi sum} of surfaces $S_1,...,S_n$, then $S$ is a fiber if and only if each of the surfaces $S_i$ is a fiber (refer to theorem \ref{Gabai}). 
 Melvin and Morton \cite{MM} studied the fibration of genus 2 knots. Goodman--Tavares \cite{GT} showed that under simple conditions imposed on certain spanning surfaces, it is possible to decide whether or not these surfaces are fibers for pretzel links. Their method is very algebraic and relies on Stallings' work \cite{St}.  Very recently Futer--Kalfagianni--Purcell (\cite{FKP1}, theorem 5.11) introduced a new method for deciding whether or not a given spanning surface is fiber for a link $K$. From a diagram of the link they construct an associated  surface (called the  $A$-state surface) and a certain graph. They show that this surface is a fiber if and only if the corresponding graph is a tree.  Later, Futer \cite{Fu} extended this result to a larger class of links (\textit{homogeneous links}) using combinatorial arguments. 

In this paper we will be mainly concerned with the fibration of  three classes of links:  \textit{augmented links, locally alternating augmented links}, and links obtained from augmented links by Dehn filling on certain components. Given the diagram of an augmented link we associate a spanning surface and a graph. We then show that this surface is a fiber for the link complement if and only if the associated graph is a tree. We also show that when this is the case, then  Dehn filling on certain components of these links produces fibered manifolds. This last result is then used to show that every locally alternating augmented link  is fibered, and explicitly exhibit their fibers. A relevant remark here is that the surfaces and graphs we consider are very different from those in \cite{FKP1} and \cite{Fu}. It is also interesting to observe that we obtain the same type of results: a link complement is fibered if and only if an associated graph is a tree. 

The relevance of studying augmented links is that they have played a central role in several recent developments in 3-manifold topology. Lackenby and Agol--Thurston \cite{La} used them to estimate volumes of alternating link complements. Futer--Kalfagianni--Purcell \cite{FKP} used them  to obtain diagrammatic volume estimates of many knots and links.  Futer--Purcell \cite{FP} also used them to prove that if $K$ is a link  with a twist-reduced diagram with at least 4 twist regions and at least 6 crossings per twist region, then every non-trivial Dehn filling on $K$ is hyperbolic.  Their combinatorial argument further implies that every link with at least 2 twist regions and at least 6 crossings per twist region is hyperbolic and gives a lower bound for the genus of $K$. Cheesebro--DeBlois--Wilton \cite{CDW} proved that hyperbolic augmented links satisfy the virtual fibering conjecture.  

We next define the classes we will be working with and state the main results.

\section*{acknowledgements}
I am  very grateful to  Alan Reid for his guidance during my graduate program.  I am also thankful  to Cameron Gordon for helpful conversations and Jo\~ao Nogueira and Jessica Purcell for their comments on an early draft of this work.  Finally I would like to thank the referee for his careful reading of this paper and his many comments which helped improve it.


\section{Augmented links,  locally alternating augmented links and main results}
\label{section_thms}

The notion of \textit{augmented links} was first introduced by Adams \cite{Ad1}  and further explored by Futer--Kalfagianni--Purcell \cite{FKP}, Futer--Purcell \cite{FP} and Purcell \cite{Pu1}. We recall it here. For more details see the very nice survey paper on augmented links by Purcell \cite{Pu}. 

Let $K$ be a link in $S^3$ with diagram $D(K)$. Regard $D(K)$ as a $4$-valent graph in the plane. A \textit{bigon region} is a complementary region of the
graph having two vertices in its boundary. A string of bigon regions of the
complement of this graph arranged end to end is called a \textit{twist region}. A
vertex adjacent to no bigons will also be a twist region. Encircle each twist
region with a single unknotted component, called a \textit{crossing circle}, obtaining
a link $J$. $S^3-J$ is homeomorphic to the complement of the link $L$ obtained
from $J$ by removing all full twists from each twist region. The link $L$ is
called the augmented link associated to $D(K)$.  The original link complement can be obtained from the link $L$ by performing $1/n$-Dehn
filling on the crossing circles, for  appropriate choices of $n$.

\begin{figure}[h]
\begin{center}
\includegraphics[scale=.09]{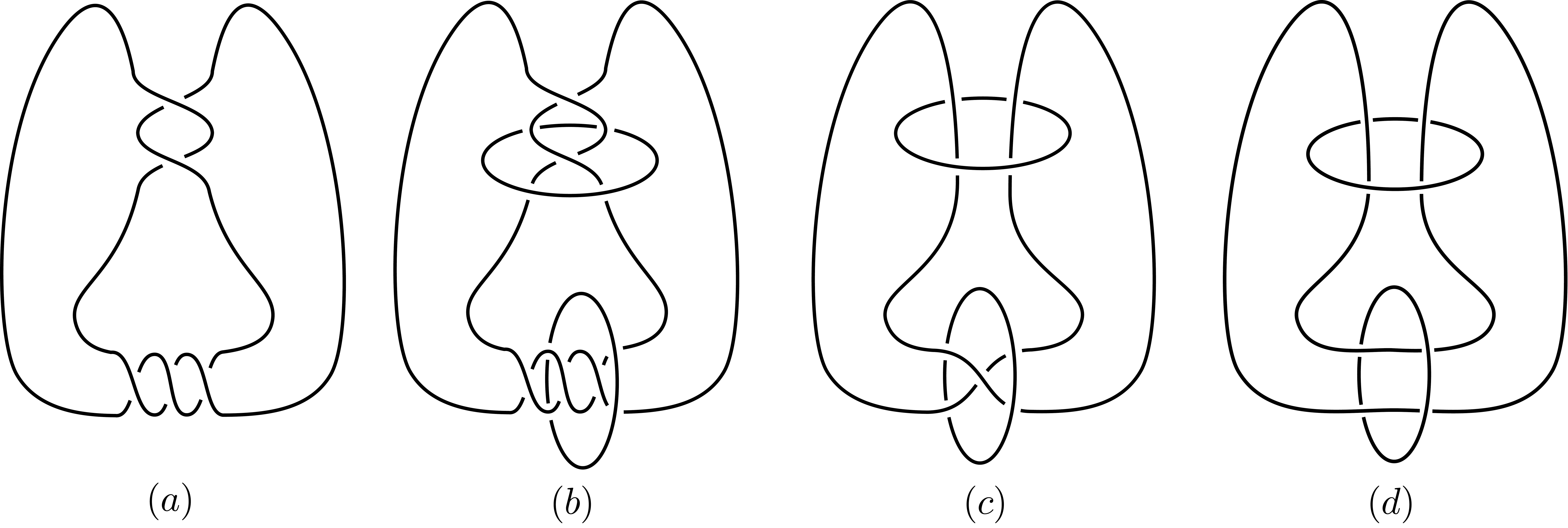}
\caption{(a) Initial link $K$; (b) link $J$ obtained by adding crossing circles; (c) augmented link $L$; (d) corresponding flat augmented link.}
\label{augmenting}
\end{center}
\end{figure}

When all the twist regions in the diagram $D(K)$ have an even number of crossings, then all non-crossing circle components of the augmented link $L$ will be embedded in the projection plane. We call these links  \textit{flat augmented links}.

 Given the diagram of a flat augmented link $L$ we construct a  Seifert surface $S_L$, called \textit{standard Seifert surface},  and a graph $G_B(L)$ (this will be done in section \ref{section_set_up}). We now state our main results.

\begin{thm}\label{main}
Let $L$ be a flat augmented link.  Then the standard Seifert surface $S_L$ is a fiber if and only if  the graph $G_B(L)$ is a  tree. 
\end{thm}

Performing $\pm 1$ Dehn filling on crossing circle components of links as above yield new ones which are again fibered. 

\begin{thm}\label{thm_surgery}
Let $L$ be an flat augmented link such that the graph $G_B(L)$ is a tree. Then  $\pm 1$ Dehn filling on crossing circle components yields a fibered link $K$.
\end{thm}
Given a flat augmented link $L$, one can construct a corresponding \textit{locally alternating augmented link}  $L_a$ as follows: in each crossing circle change two of the crossings  so that the crossings in the crossing circles are  alternating. This is described in Figure \ref{L_a}. Note that the resulting link need not to be alternating. 

\begin{figure}[h]
\begin{center}
\includegraphics[scale=.12]{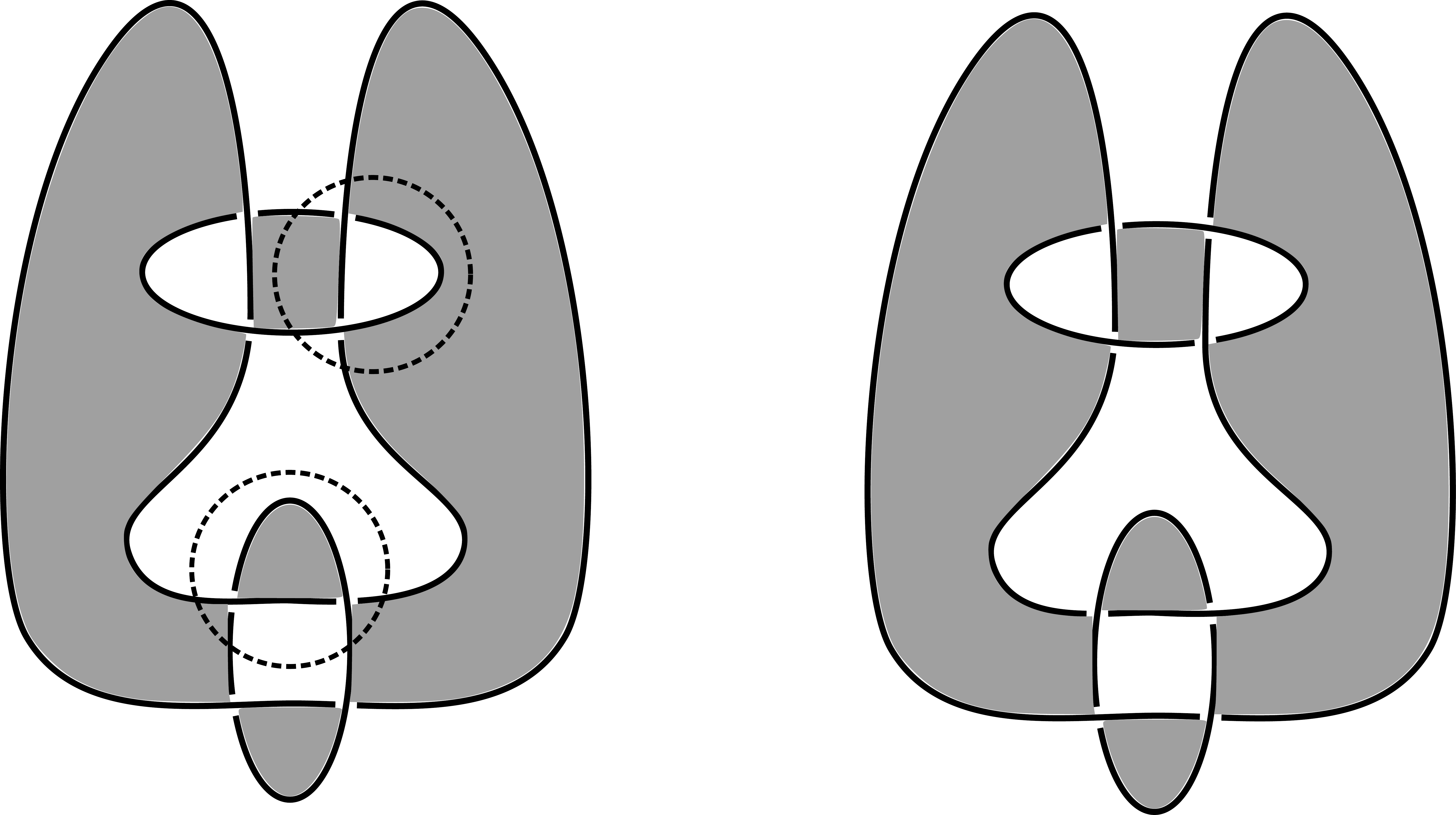}
\caption{Left: Pairs of crossings yielding  locally alternating link; Right: Seifert surface for resulting link.}
\label{L_a}
\end{center}
\end{figure}

We later show that every locally alternating augmented link $L_a$ can be obtained from $\pm 1$ filling  on the crossing circles of a flat augmented link $\tilde{L}$ such that $G_B(\tilde{L})$ is a tree. This implies

\begin{thm}\label{thm_alt}
Let $L_a$ be a locally alternating augmented link obtained from a flat augmented link $L$ with a connected diagram.  Then $L_a$ fibers. 
\end{thm} 

\begin{remark}
 We remark that our methods and  the surfaces and graphs we construct are very different from those in \cite{FKP1} or even \cite{Fu}. However, it is very interesting that we are obtaining the same type of results: a manifold fibers given that a certain associated graph is a tree. We also note that very often fibration of the links considered here cannot be detected from their construction but is detected by ours (examples for the converse can also be exhibited).  
\end{remark}

The remainder of the  paper is organized as follows: in section \ref{Murasugi sum} we  we recall the notion of \textit{Murasugi sum}. This will be used in the decomposition of the \textit{standard surface}. In section \ref{section_set_up} we set up the terminology of standard surface and of the graph $G_B(L)$  needed for the remainder of the paper.  In section \ref{section_proof} we prove theorem \ref{main}. In section \ref{section_surgery} we prove theorem \ref{thm_surgery} . Finally, in section \ref{section_alt} we use theorem \ref{thm_surgery}  to prove theorem \ref{thm_alt}.


\section{Murasugi Sum}\label{Murasugi sum}

In this section we recall the notion of Murasugi sum (\cite{Ga1} and \cite{Mu}).

\begin{definition}
We say that the oriented surface $T$ in $S^3$ is with boundary $L$ is the Murasugi sum of the two oriented surfaces $T_1$ and $T_2$ with boundaries $L_1$ and $L_2$ if there exists a $2$-sphere $S$ in $S^3$  bounding  balls $B_1$ and $B_2$ with $T_i\subset B_i$ for $i=1,2$, such that $T=T_1\cup T_2$ and $T_1\cap T_2=D$ where $D$ is a $2n$-sided disk contained in $S$ (see Figure \ref{Murasugi_sum}).  
\end{definition}

\begin{figure}[h]
\begin{center}
\includegraphics[scale=.18]{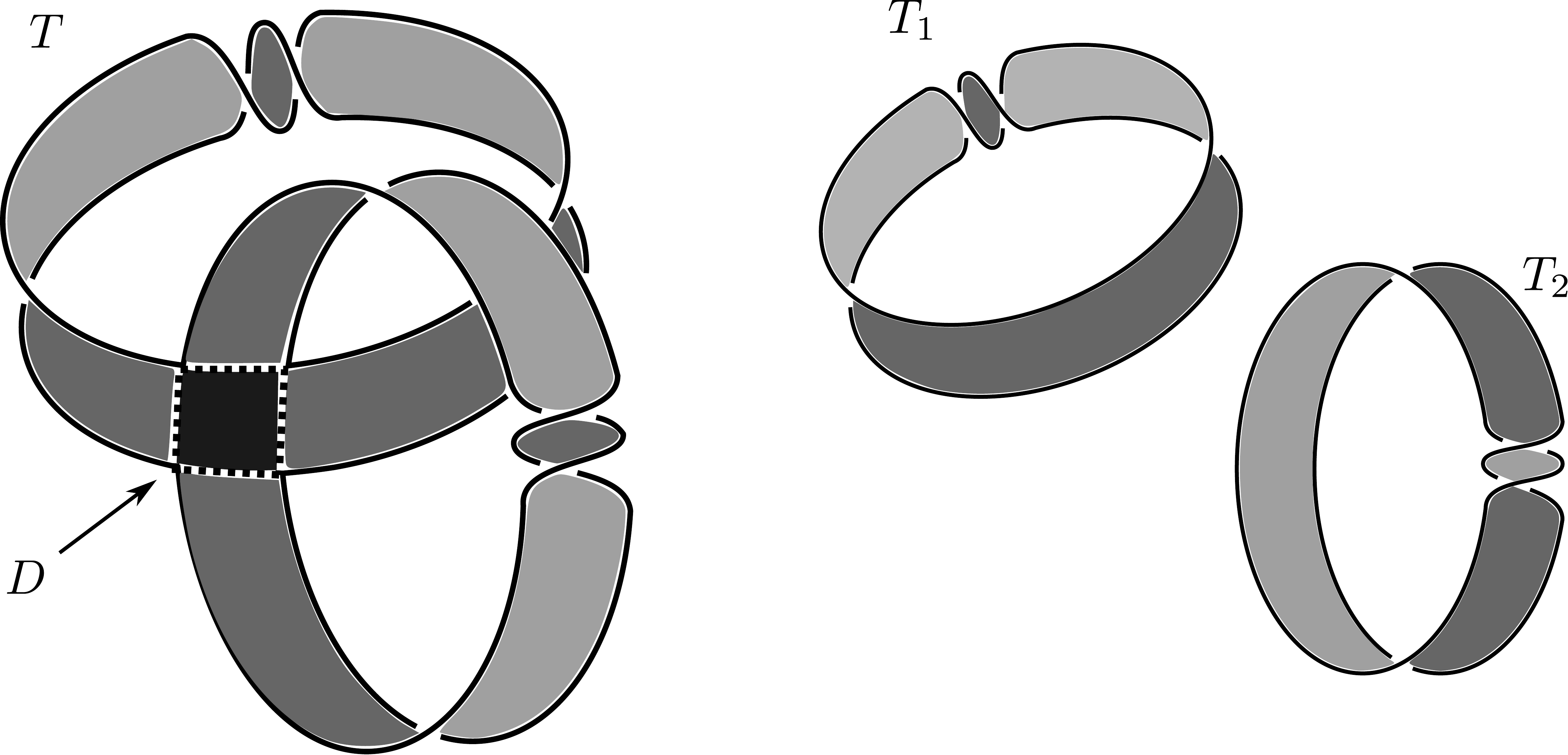}
\caption{Murasugi sum of two Hopf bands}
\label{Murasugi_sum}
\end{center}
\end{figure}

The result concerning Murasugi sum we need is the following, due to Gabai \cite{Ga}.

\begin{thm}[Gabai]\label{Gabai}
Let $T\subset S^3$, with $\partial T=L$, be a Murasugi sum of oriented surfaces $T_i\subset S^3$, with $\partial T_i= L_i$, for $i=1,2$.  Then $S^3-L$ is fibered with fiber $T$ if and only if $S^3-L_i$ is fibered with fiber $T_i$ for $i=1,2$. 
\end{thm}

\begin{remark}
We abuse notation by saying that $L$ is the Murasugi sum of $L_1$ and $L_2$.
\end{remark}


\section{Set up}\label{section_set_up}

The augmented links we consider are flat, i.e., there are no twists adjacent to crossing circles. 
Our goal is to find a condition under which such  links fiber. Recall that Stallings \cite{St} provided a method for checking whether or not a given Seifert surface is a fiber for the complement of an oriented link.

\begin{thm}[Stallings]\label{stallings}
Let $T\subset S^3$ be a compact, connected, oriented surface with nonempty boundary $\partial T$. Let $T\times[-1,1]$ be a regular neighborhood of $T$ and let $T^+=T\times\{1\}\subset S^3-T$. Let $f=\varphi|_{T}$, where $\varphi:T\times[-1,1]\longrightarrow T^+$ is the projection map. Then $T$ is a fiber for the link $\partial T$ if and only if  the induced map $f_*:\pi_1(T)\longrightarrow\pi_1(S^3-T)$ is an isomorphism.   
\end{thm}

Consider the diagram of $L$ as a planar graph. This graph divides the plane into regions which are checkerboard colored. The unbounded region is colored white and the other regions are colored accordingly. Black regions correspond to a Seifert surface, called \textit{standard  surface}, which induces an orientation on the link $L$. This surface will be denoted $S_{L}$. There are three types of white regions: 

\begin{itemize}
\item[] \textbf{Type $A$ regions}, bounded by  crossing circles that bound two white regions; 
\item[] \textbf{Type $B$ regions}, bounded by  crossing circles that bound a single white region;
\item[] \textbf{Type $C$ regions}, not bounded by crossing circles.
\end{itemize}

\begin{figure}[h]
\begin{center}
\includegraphics[scale=.12]{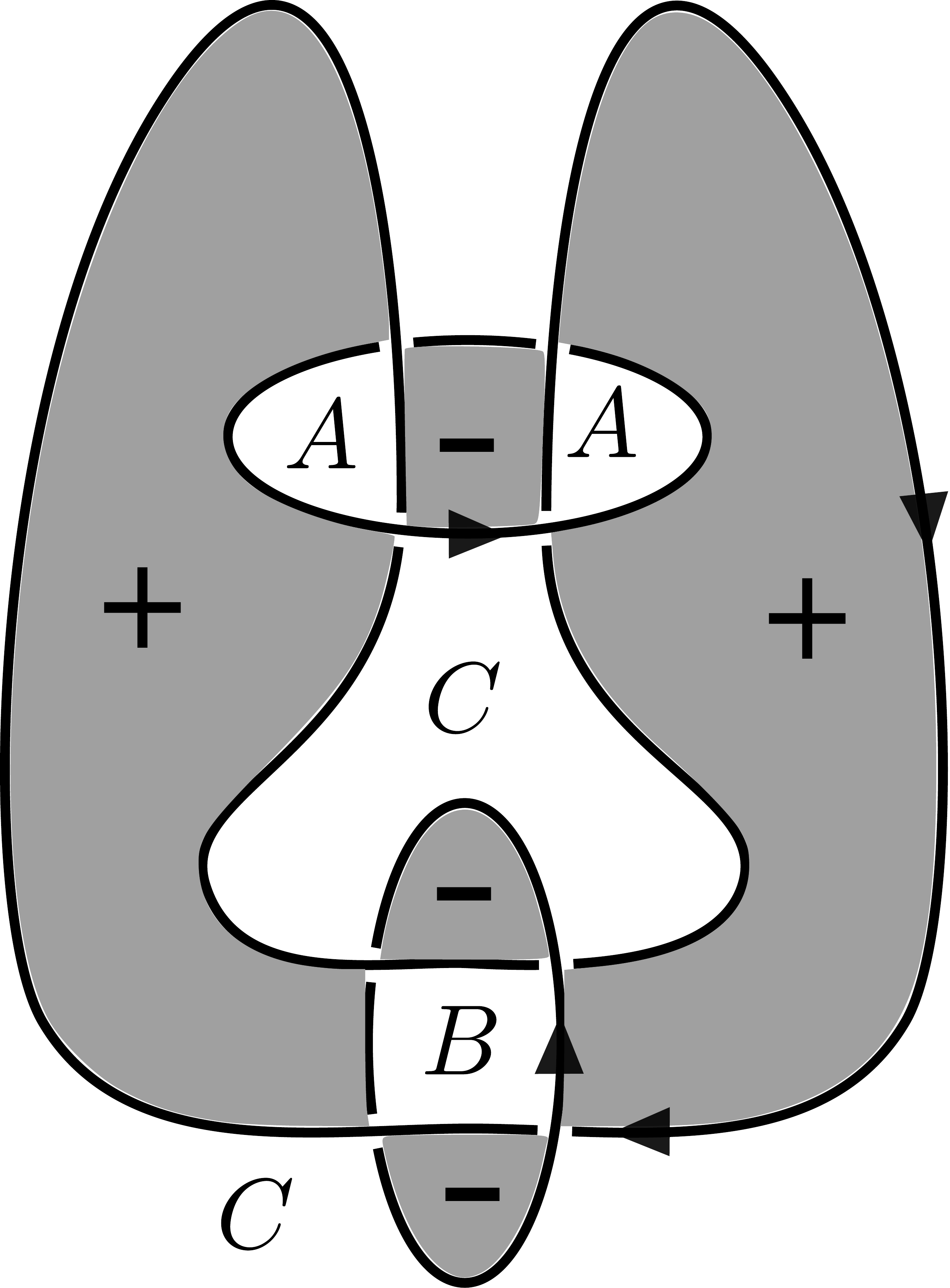}
\caption{Standard Seifert surface for augmented link and corresponding white regions determined by the diagram}
\label{regions}
\end{center}
\end{figure}

Note that type $A$ regions come in pairs. We will the denote the above regions by $A_{11}, A_{12}, A_{21}, A_{22},..., A_{p1}, A_{p2}, B_1,..., B_q, C_1,..., C_r$ respectively. We  denote the unbounded white region by $C_0$.

A crossing circle bounding a type $B$ region $B_i$ will be called a \textit{$B$-circle}, also denoted by $B_i$. Note that given the diagram for $L$, as described above, there will be two type $C$ regions  adjacent to each $B$-circle.  A crossing circle bounding a pair of type $A$ regions $A_{i1},A_{i2}$ will be called a \textit{$A$-circle} and denoted by $A_i$. Note that there will be two type $C$ regions adjacent to each $A$-circle. 
Let $G_B(L)$ be the graph obtained from the diagram of $L$ as follows (see figure \ref{Newlink}). 

\begin{itemize}
\item[] \textbf{Vertices} are type $C$ regions; 
\item[] An \textbf{edge} joins $C_i$ and $C_j$ if there is a $B$-circle adjacent to them.
\end{itemize}

\begin{figure}[h]
\begin{center}
\includegraphics[scale=.10]{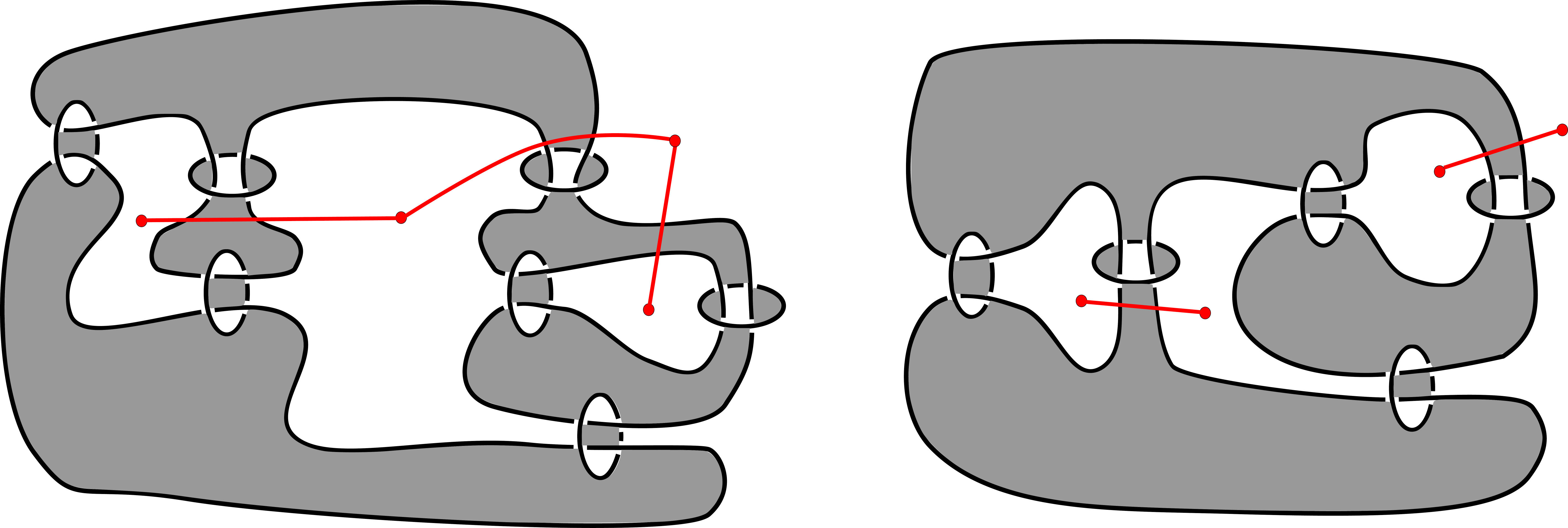}
\caption{Left: Graph $G_B(L)$ associated to the standard surface is a tree; Right: $G_B(L)$ is disconnected.}
\label{Newlink}
\end{center}
\end{figure}
 
 The standard surface $S_{L}$ can be obtained as the Murasugi sum of a surface $S_{L'}$ and a collection of Hopf bands. We can remove  each $A$-circle in the diagram by deplumbing (i.e., undoing Murasugi sum) a pair of Hopf bands. This is described in figure \ref{decompose}. Applying this procedure to every $A$-circle we obtain a new augmented link $L'$, with standard surface $S_{L'}$.  This decomposition will be  used in the proof of theorem \ref{main}. 
 
\begin{figure}[h]
\begin{center}
\includegraphics[scale=.080]{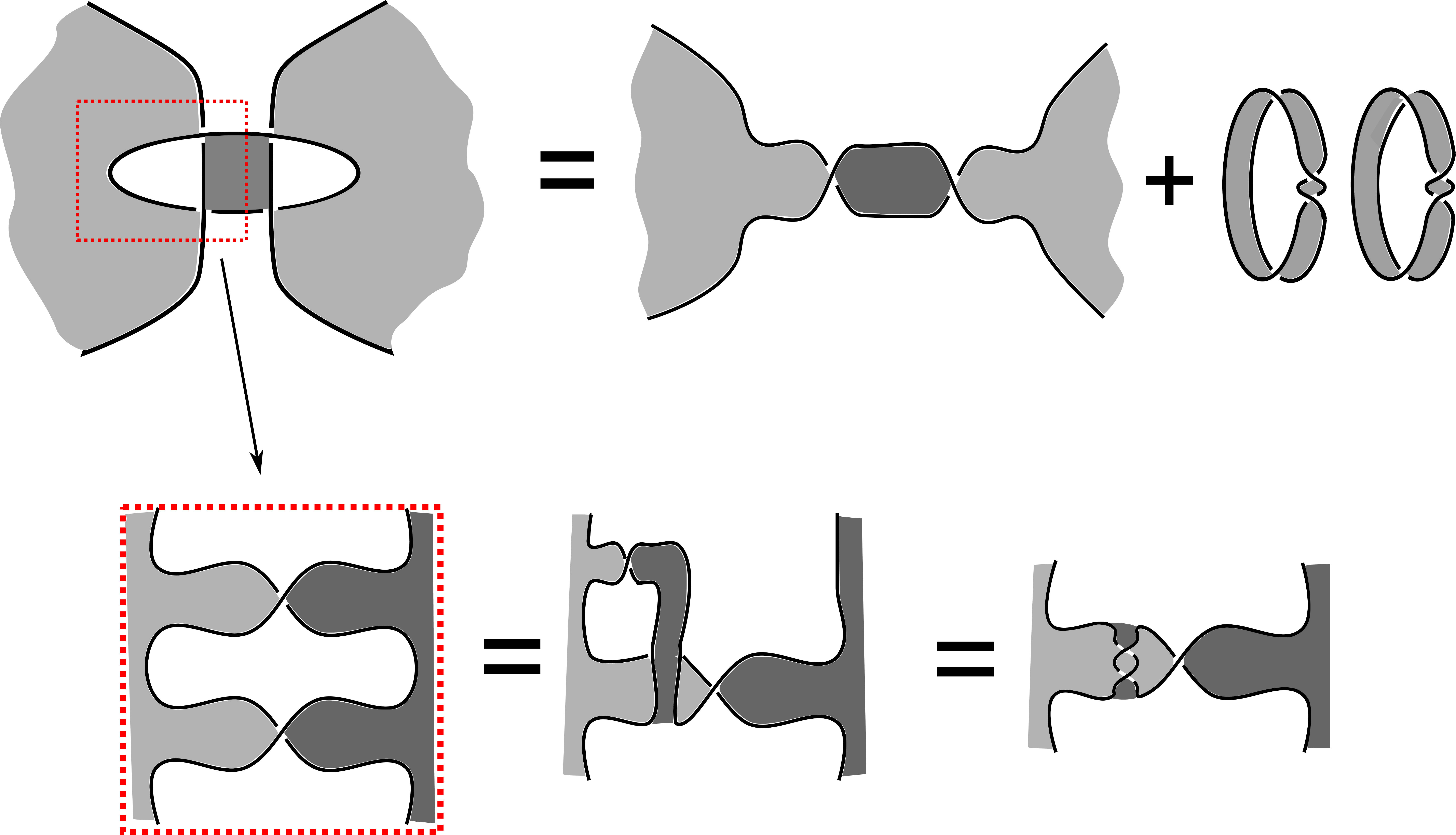}
\caption{Decomposing a pair of Hopf bands for each $A$-circle.}
\label{decompose}
\end{center}
\end{figure}

\section{Proof of theorem \ref{main}}\label{section_proof}

\begin{main}
Let $L$ be a flat augmented link.  Then the standard Seifert surface $S_L$ is a fiber if and only if  the graph $G_B(L)$ is a  tree. 
\end{main}

The standard surface $S_L$ is obtained as the Murasugi sum of  a collection of Hopf bands together with the surface  $S_{L'}$. Each of these Hopf bands is a fiber for the complement of their boundary Hopf links.  The link  $L'$ is  the boundary  of $S_{L'}$. $L'$ is itself a flat augmented link in which all crossing circles are $B$-circles. The surface $S_{L'}$ is two-sided (i.e., orientable): from above the diagram  we see the  \textquotedblleft $-$\textquotedblright side of the surface on the black regions bounded by $B$-circles. The other black regions represent the \textquotedblleft $+$\textquotedblright side. We use this orientation throughout the paper.

  Observe that, by construction, $G_B(L)=G_B(L')$. In view of Theorem \ref{Gabai} we need to find conditions under which the surface $S_{L'}$ is a fiber for the link $L'$. 
This is given by the following.  

\begin{lem}\label{fib}
The oriented link $L'$ fibers with fiber $S_{L'}$ if and only  if its associated graph $G_B(L')$ is a tree.  
\end{lem} 
This lemma concludes the proof of Theorem {\ref{main}}.

\begin{figure}[h]
\begin{center}
\includegraphics[scale=.10]{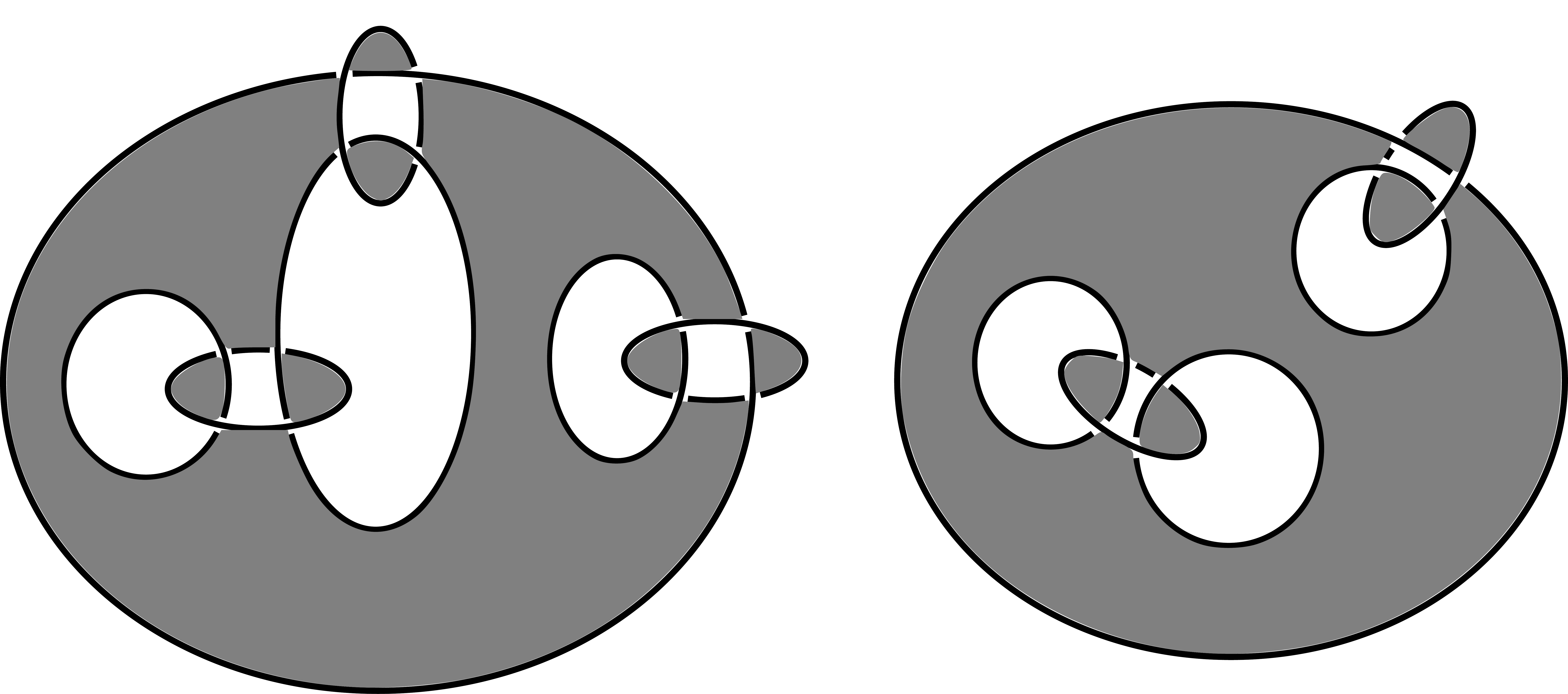}
\caption{Surfaces $S_{L'}$ obtained from the standard surfaces in figure \ref{Newlink}}.
\label{S_L'}
\end{center}
\end{figure}

We now proceed to prove the lemma.

\begin{proof}
First observe that the link $L'$ is itself a flat augmented link, obtained from $L$ by decomposing a pair of Hopf bands for each  $A$-circle. The diagram of $L'$ divides the plane into type $B$ and type $C$ regions. 

Observe also that the fundamental group of the surface $S_{L'}$ is free.   Let $B_1,...,B_q$ and $C_1,...,C_r$ be the type $B$ and $C$ regions given by the diagram of $L'$. A generating set for $\pi_1(S_{L'})$ is given as follows. Consider simple closed curves  on $S_{L'}$ around these regions. They are labeled $\alpha_{b_1},...\alpha_{b_q}, \alpha_{c_1},...,\alpha_{c_r}$ and are oriented in the  counter-clockwise direction. Now choose a base point $a\in S_{L'}$   in such a way that, when seen from above the projecting plane, it lies in the \textquotedblleft $+$\textquotedblright side of $S_{L'}$. Finally, add  arcs $h_{b_1},...,h_{b_q}, h_{c_1},...,h_{c_r}$  from $a$ to each of the curves above. This gives loops $\beta_{b_i}=h_{b_i}\alpha_{b_i}h_{b_i}^{-1}$, $\beta_{c_j}=h_{c_j}\alpha_{c_j}h_{c_j}^{-1}$ based at $a$.  This set of based loops corresponds to a generating set  for $\pi_1(S_{L'})$. These generators will  be denoted by $u_{b_1},...,u_{b_q}$, $u_{c_1},..., u_{c_r}$.

The fundamental group  of the complement $S^3-S_{L'}$  is also free. We now describe a generating set for this group. As before, let $C_0$ denote the unbounded white region determined by the diagram of $L'$. Let $B_i$ denote a white region determined by a $B$ circle and let $\gamma_{b_i}\subset S^3-S_{L'}$ be a semicircle   with one endpoint in $C_0$ and the other in $B_i$,  lying under the projecting plane. Associated to each $B$ region  we construct a simple closed curve by connecting the endpoints of the arc $\gamma_{b_i}$ to the point $f(a)$ by straight line segments. Here, $a$ is the base point for $\pi_1(S_{L'})$ and $f:S_{L'}\longrightarrow S^3-S_{L'}$ is described in theorem \ref{stallings}. Each of these curves is oriented so that, starting at $f(a)$, we move along the line segment connecting $f(a)$ to the endpoint of $\gamma_{b_i}$ in $B_i$, then move along $\gamma_{b_i}$ to the second endpoint and then back to $f(a)$ through the second line segment. Make a similar construction for each type $C$ region.  We have built  loops with base point $f(a)$ corresponding to a set of generators for $\pi_1(S^3-S_{L'})$.   These generators are denoted   by  $x_{b_1},...,x_{b_q}, x_{c_1},..., x_{c_r}$, according to the type of region they cross.

We can now  describe the induced map $f_*$ in terms of these generators. If $\beta$ is a loop in $S_{L'}$  based at $a$ then $f(\beta)$ is a loop in $S^3-S_{L'}$ based at $f(a)$. If $u\in\pi_1(S_{L'})$ represents the homotopy class of $\beta$, then $f_*(u)$ is an element in $\pi_1(S^3-S_{L'})$ given  as a word in the letters  $x_{b_1},...,x_{b_q}, x_{c_1},..., x_{c_r}$ as follows: start at $f(a)$ and move along $f(\beta)$. If it crosses a white region other than $C_0$ from above to below the projecting plane, we write the corresponding letter $x\in\{x_{b_1},...,x_{b_q}, x_{c_1},..., x_{c_r}\}$. If it crosses the white region from below to above the projecting plane, we write the letter $x^{-1}$. If it crosses $C_0$ we do not write any letters. Going around $f(\beta)$ once gives the desired word.

Now let $\beta_{b_1},...,\beta_{b_q}, \beta_{c_1},...,\beta_{c_r}$ be the based loops corresponding to the generators of $\pi_1(S_{L'})$ as above. They were given by  $\beta_{b_i}=h_{b_i}\alpha_{b_i}h_{b_i}^{-1}$, $\beta_{c_j}=h_{c_j}\alpha_{c_j}h_{c_j}^{-1}$.
 
Let $B_i, C_n, C_m$ be  white regions, $C_n, C_m$ adjacent to the $B$-circle $B_i$, as in figure \ref{Bcircle}. The loops in figure \ref{Bcircle} (left) represent  the simple closed curves  in $S_{L'}$ around these regions. There, we indicate the corresponding generators of $\pi_1(S_{L'})$.  Their images under $f_*$ are given by   
$$\left\{\begin{array}{lr}
u_{b_i}\mapsto  g_{b_i}x_{c_m}x_{c_n}^{-1}g_{b_i}^{-1}\\
u_{c_m}\mapsto g_{c_m}x_{b_i}x_{c_m}^{-1}w_mg_{c_m}^{-1}, \hspace{.2cm}\text{$w_m$ is a word without the letter $x_{b_i}$};\\
u_{c_n}\mapsto g_{c_n}x_{c_n}x_{b_i}^{-1}w_ng_{c_n}^{-1}, \hspace{.2cm}\text{$w_n$ is a word without the letter $x_{b_i}$}. 
\end{array}\right\}
$$

Here  $g_{b_i}, g_{c_j}$ are the words we get in the letters $\{x_{b_1},...,x_{b_q}, x_{c_1},..., x_{c_r}\}$ by moving along the arcs $h_{b_i}, h_{c_j}$ respectively. Thus, up to conjugation, we get 
\begin{equation}
\left\{\begin{array}{lr}
u_{b_i}\mapsto x_{c_m}x_{c_n}^{-1}\\
u_{c_m}\mapsto x_{b_i}x_{c_m}^{-1}w_m, \hspace{.2cm} \text{$w_m$ is a word without the letter $x_{b_i}$};\\
u_{c_n}\mapsto x_{c_n}x_{b_i}^{-1}w_n, \hspace{.2cm}\text{$w_n$ is a word without the letter $x_{b_i}$}. 
\end{array}\right\}
\label{f_*}
\end{equation}

\begin{remark}
\begin{itemize}\item[]	
 \item[(a)] If a $B$-circle is adjacent to the unbounded region $C_0$ and another region $C_m$, then, up to conjugation, we have $u_{b_i}\mapsto x_{c_m}^{\pm 1}$.
\item[(b)] We may assume the words $w_m, w_n$ do not have the letter $x_{b_i}$. If they did, this would imply the simple closed curves $\alpha_{c_m},\alpha_{c_n}$ around $C_m$ and $C_n$ are the same curve, which would mean $C_m$ and $C_n$ represent the same region. In this case the circle $B_i$ would not be linked to the other components of  $L'$, i.e., $L'$  would be a split link, which is not fibered.
\end{itemize}
\end{remark}

\begin{figure}[h]
\begin{center}
\includegraphics[scale=.17]{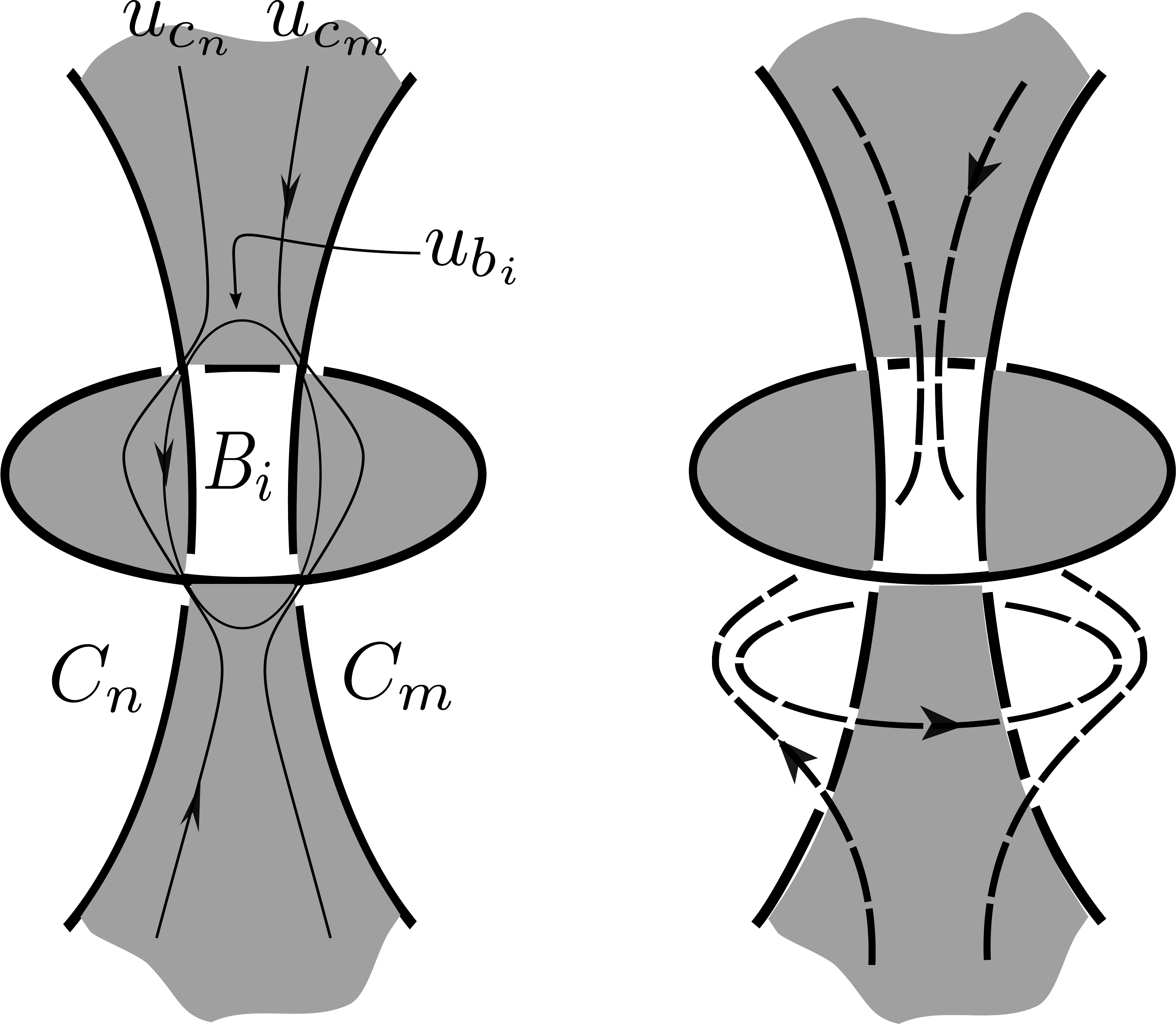}
\caption{Left: Loops around white regions corresponding to generators $u_{b_i},u_{c_m}$ and $u_{c_n}$; Right: Their images under $f$.}
\label{Bcircle}
\end{center}
\end{figure}



The strategy now is as follows: by theorem \ref{stallings}  $S_{L'}$ is a fiber if and only if the map $f_*:\pi_1(S_{L'})\longrightarrow\pi_1(S^3-S_{L'})$ is an isomorphism.  
We show that if $G_B(L')$ is not a  tree  then the map induced by  $f_*$ on homology is not an isomorphism. Therefore $f_*$ cannot be isomorphism.  When $G_B(L')$ is a  tree we prove that  $f_*$ is surjective. Since free groups have the Hopfian property, it follows that $f_*$ is an isomorphism.

First assume $G_B(L')$ is not a tree

Consider the map $\bar{f}_*:H_1(S_{L'})\longrightarrow H_1(S^3-S_{L'})$ induced on homology by $f_*$. Denote by $\bar{u}_{b_1},...,\bar{u}_{b_q}$, $\bar{u}_{c_1},..., \bar{u}_{c_r}$ the generators of $H_1(S_{L'})$, corresponding to the above generators of $\pi_1(S_{L'})$. The generators of $H_1(S^3-S_{L'})$ are constructed similarly and denoted $\bar{x}_{b_1},...,\bar{x}_{b_q}$, $\bar{x}_{c_1},..., \bar{x}_{c_r}$.

\textbf{Case 1.} \underline{\textit{$G_B(L')$ has two or more connected components}}:

Let $C_1,...,C_{n+1}$ represent the vertices of a component not containing the vertex $C_0$.
Let $\beta\subset S_{L'}$ be a simple closed curve around the regions $C_1,...,C_{n+1}$. This loop is obviously non-trivial in $S_{L'}$.  Moreover, the corresponding element $\bar{u}\in H_1(S_{L'})$ is also non-trivial. However we see that $f(\beta)$ is a trivial loop in $S^3-S_{L'}$ and therefore $\bar{f}_*(\bar{u})\in H_1(S^3-S_{L'})$ is trivial.  

\textbf{Case 2.} \underline{\textit{$G_B(L')$ contains a non-trivial loop}}: Let $C_1,...,C_n,C_{n+1}=C_1$ represent the vertices and $B_1,...,B_n$ the edges of this loop. In (\ref{f_*}) above we have expressions for $f_*$ up to  conjugation. Therefore, at the level of homology, we have  

$$\bar{f}_*(\bar{u}_{b_i})= -\bar{x}_{c_i}+\bar{x}_{c_{i+1}}$$ 
and thus 
$$\bar{f}_*(\bar{u}_{b_1}+\cdots+ \bar{u}_{b_n})=(-\bar{x}_{c_1}+\bar{x}_{c_2})+\cdots +(-\bar{x}_{c_n}+\bar{x}_{c_1})=0$$
The same argument holds if the loop contains the vertex corresponding to $C_0$.

Suppose now  $G_B(L')$ is a  tree. In this case note that there is only one \textquotedblleft +\textquotedblright region seen from above the projecting plane. We  take the base point $a\in S_{L'}$  lying in this region. This choice of base point allows us  to choose the arcs $h_{b_i},h_{c_j}$ in $S_{L'}$  joining $a$ to the simple closed curves around the white regions in such a way that all these arcs lie in the \textquotedblleft +\textquotedblright region of $S_{L'}$. In this situation, the expressions in (\ref{f_*}) give the actual images under  $f_*$  of the generators of $\pi_1(S_{L'})$.  We  now show $f_*$ is surjective. 

\begin{step}\label{step1}
\end{step} Let $C_0$ be the initial vertex and let $C_1,...,C_k$ be the vertices connected to $C_0$ by edges represented by the $B$-circles $B_1,...,B_k$. We have  
$$u_{b_i}\mapsto x_{c_i}^{\pm 1}$$
Thus $x_{c_i}$ is  the image of either $u_{b_i}$ or $u_{b_i}^{- 1}$.

\begin{step}\label{step2}
\end{step} Let $C_{11},...,C_{1k_1}$ be connected to $C_1$ by edges $B_{11},...,B_{1k_1}$. Note that, since $G_B(L')$ is a tree, $C_{1i}\neq C_j$ for $i=1,...,k_1$ and $j=0,...,k$. We thus have 
$$u_{b_{1j}}\mapsto x_{c_1}^{\pm 1}x_{c_{1j}}^{\mp 1}$$
And hence 
$$u_{b_1}^{\mp 1}u_{b_{1j}}\mapsto x_{c_{1j}}^{\mp 1}$$
This determines  preimages  for $x_{c_{1i}}$ in terms of $u_{b_{1i}}$ and $u_{b_1}$.

\begin{step}\end{step}
 Repeat Step 2 for $C_2,...,C_k$. These $C_i$'s are the ones defined in Step \ref{step1}. This determines preimages  for $x_{c_{ji}}$ in terms of $u_{b_{ji}}$ and $u_{b_j}$.

\begin{step}
\end{step} Inductively find preimages  for all the $x_c's$ in terms of the $u_b's$. Note that, for this step to work, it is necessary that 
$G_B(L')$ is a tree. 

Now we need to find preimages for the $x_b$'s.

\begin{step}\end{step} From 
$$u_{c_m}\mapsto x_{c_m}^{-1}x_{b_i}w_m, \text{$w_m$ word without the $x_{b_i}$ letter}$$
$$u_{c_n}\mapsto x_{c_n}x_{b_i}^{-1}w_n, \text{ $w_n$ word without the $x_{b_i}$ letter}$$
we show how to obtain the $x_b$'s as images of $u_b$'s and $u_c$'s.

Since $G_B(L')$ is a tree, each $C$-region corresponding to a terminal vertex is adjacent to a single $B$-circle. Suppose $C_m$ is such a region.   In this case  we have  $u_{c_m}\mapsto x_{c_m}^{-1}x_{b_i}$ and thus we  obtain $x_{b_i}$ as the image of a word in the letters $u_b's$ and $u_c's$ (and their inverses). Repeat this process for all the terminal vertices. Inductively we can  find all the $x_b$'s as the image of a word in the letters $u_b$'s and $u_c$'s.  
\end{proof}


\section{Dehn filling on crossing circles}\label{section_surgery}

In this section we prove

\begin{thm_surgery}
Let $L$ be an flat augmented link such that the graph $G_B(L)$ is a tree. Then $\pm 1$ Dehn filling on crossing circle components yields a fibered link $K$.
\end{thm_surgery}

\begin{figure}[h]
\begin{center}
\includegraphics[scale=.1]{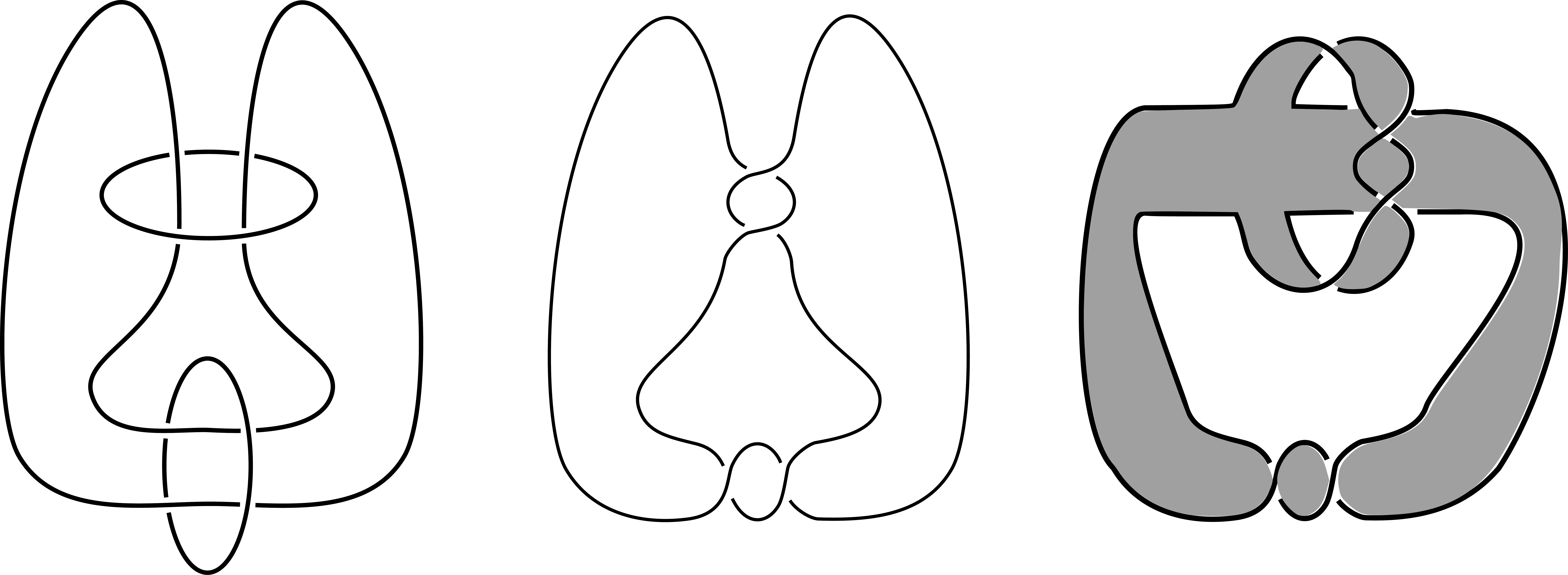}
\caption{Left: Original link; Center: $\pm1$ filling on crossing circles; Right: Standard Seifert surface for resulting link.}
\label{surgery}
\end{center}
\end{figure}

\begin{proof}

 $\pm 1$ filling on a $A$-circles yields a fibered link. To see this consider the link $L'$ constructed as in the last section. Recall that when $L$ fibers with fiber $S_L$, then $L'$ fibers with fiber $S_{L'}$. $A$-circles in $L$ correspond to the Murasugi sum of two Hopf bands with $L'$. It should be easy to see that $\pm 1$ filling on a $A$-circle corresponds to the sum  of a single Hopf band (see figure \ref{surgery}). Therefore, by theorem \ref{Gabai}, the statement is true for filling on  $A$-circles. 

 The proof of  the statement for $B$-circles is similar to the proof that $S_{L'}$ is a fiber for $L'$. Let $K'$ be the link obtained from $L'$ by performing $\pm 1$ Dehn filling on the $B$-circles. We consider the standard Seifert surface $S_{K'}$ obtained from the Seifert algorithm applied to the link $K'$ (see Figure \ref{Bsurgery}).  This surface is the surface obtained by checkerboard coloring the regions determined by the diagram of $K'$ in the projecting plane (this is a 4-valent diagram). We now describe the map $f_*:\pi_1(S_{K'})\longrightarrow \pi_1(S^3-S_{K'})$. First we introduce some conventions.  We may consider the simple closed curves around white regions of $S_{K'}$,  oriented counter-clockwise. As before, since $G_B(L)$ is a tree, we only see a single \textquotedblleft +\textquotedblright region on $S_{K'}$.  Take  a base point $a\in S_{K}'$ in this region and add an arc from this base point to each of the simple closed curves. We have built a set of based loops corresponding to generators for $\pi_1(S_{K'})$. Denote this set by $\{u_c\}$.  Similar to the construction in section \ref{section_proof},  we have a corresponding set of generators $\{x_c\}$ for $\pi_1(S^3-S_{K'})$. We see that we may associate to $S_{K'}$ a tree $G_B(K')$ identical to $G_B(L')$. Vertices correspond to white regions and two vertices are joined by an edge if they are adjacent to a pair of crossings obtained by filling on a $B$-circle (see Figure \ref{Bsurgery}). The \textit{node} of the tree is the vertex corresponding to the unbounded region $C_0$. We say this vertex is at \textit{level $0$}. Let $C_{11},...C_{1s}$ be the vertices adjacent to $C_0$. We say these edges are at \textit{level $1$}. Recursively  define vertices at higher levels. $f_*$ is described as follows:

\begin{figure}[h]
\begin{center}
\includegraphics[scale=.10]{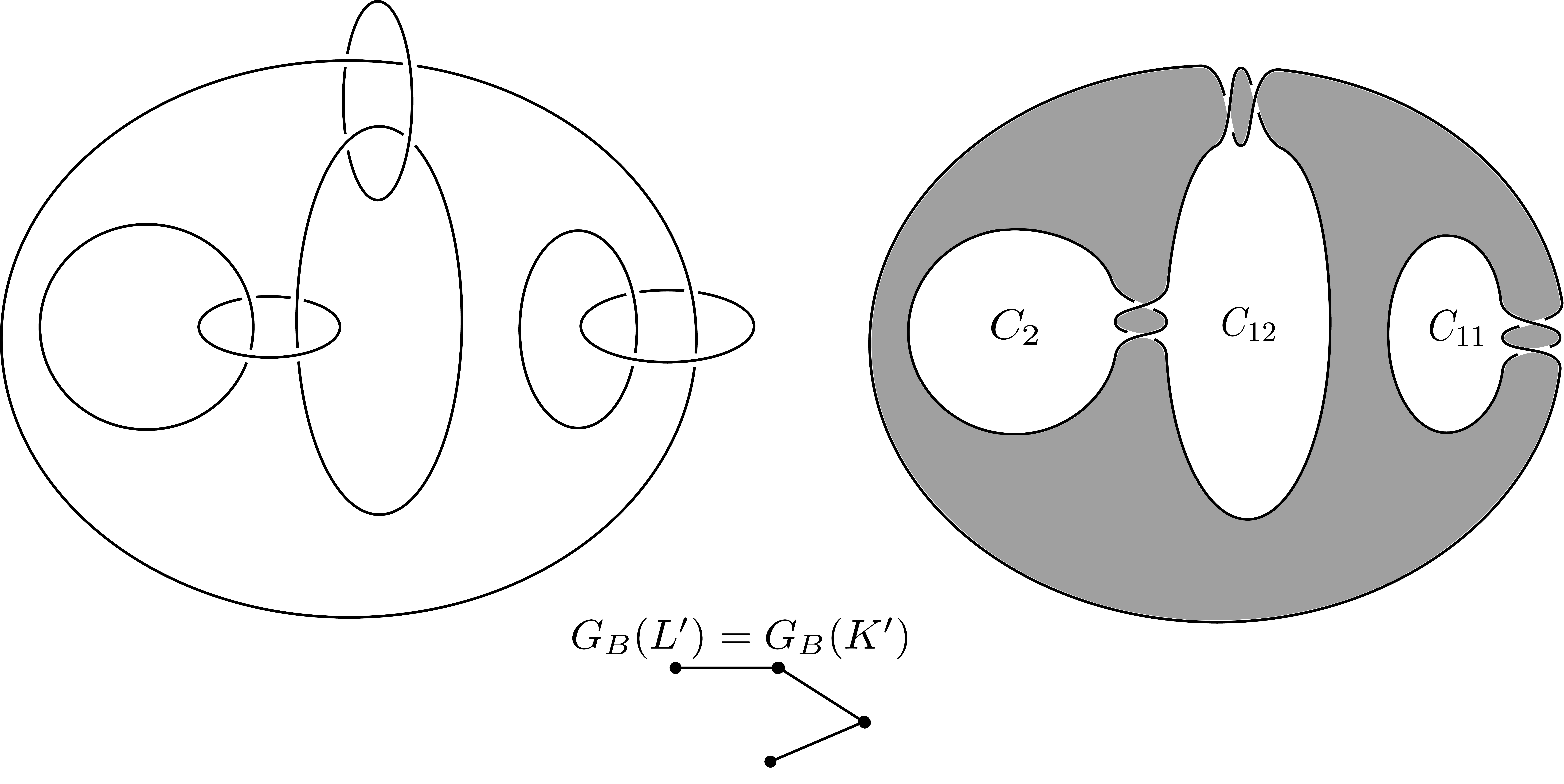}
\caption{Left: link $L'$; Right: $K'$ and its standard Seifert surface.}
\label{Bsurgery}
\end{center}
\end{figure}

\begin{itemize}
\item [1.] For a terminal vertex, say $C_m$, on $G_B(K')$ we have: $u_{c_m}\mapsto w=x_{c_m}x_{c_{m-1}}^{-1}$ or $u_{c_m}\mapsto w'=x_{c_{m-1}}x_{c_m}^{-1}$, where $C_{m-1}$ is the vertex one level lower than $C_m$ and adjacent to it. 
\item [2.] Let $C_1$ be a vertex at level 1 (adjacent to $C_0$). Suppose $C_1$ is also a terminal vertex. We have $u_{c_1}\mapsto x_{c_1}^{\pm 1}$.  
\item [2'.] If $C_1$ is not a terminal vertex, let $C_{21},\cdots,C_{2s}$ be the vertices at level 2 adjacent to it.  We have $u_{c_{1}}\mapsto x_{c_{1}}^{\pm 1}w_1w_2\cdots w_s$, where the $w_i$ are of one the forms $w_i=x_{c_{2i}}x_{c_{1}}^{-1}$ or $w_i=x_{c_{1}}x_{c_{2i}}^{-1}$.
\item [3.] An intermediate vertex (not a terminal vertex nor adjacent to $C_0$) $C_{k-1}$, say at level $k-1$, is adjacent to  vertices $C_{k1},\cdots, C_{ks}$ at level $k$ and to a single vertex   $C_{k-2}$ at level $k-2$.  We have $u_{c_{k-1}}\mapsto (x_{c_{k-1}}x_{c_{k-2}}^{-1})^{\pm 1}w_1w_2\cdots w_s$, where the $w_i$ are of one the forms $w_i=x_{c_{k-1}}x_{c_{ki}}^{-1}$ or $w_i=x_{c_{ki}}x_{c_{k-1}}^{-1}$.
\end{itemize}

 The set of generators for the free group $\pi_1(S_{K'})$ is $\{u_{c_{ij}}\}$ and    the set of generators for  $\pi_1(S^3-S_{K'})$ is $\{x_{c_{ij}}\}$. Note that $f_*$ maps the set $\{u_{c_{ij}}\}_c$  to $\{y_{ij}\}$, where $y_{ij}$ is a word on the letters $x_{c_{ij}}$. Given the tree structure on $G_B(K')$ and the description of $f_*$ above, it is not very hard to see  how to obtain $\{y_{ij}\}$ from $\{x_{c_{ij}}\}_c$ by a sequence of Nielsen transformations (see example below):  start at a terminal vertex $C_m$ and replace $x_{c_m}$ by $y_m$, where $y_m=x_{c_m}x_{c_{m-1}}^{-1}$ or  $y_m=x_{c_{m-1}}x_{c_m}^{-1}$. Repeat this for all terminal vertices. Using  the description of $f_*$ as above and an inductive argument (part 3 in the description of $f_*$ above)  we see how to obtain all the $y_{ij}$ by a sequence of Nielsen transformations.   Therefore the map $f_*:\pi_1(S_{K'})\longrightarrow \pi_1(S^3-S_{K'})$ is an isomorphism. 
\end{proof}

\begin{example} For the example in Figure \ref{Bsurgery} we have: 
$$u_{c_{11}}\mapsto x_{c_{11}}, 
u_{c_{12}}\mapsto x_{c_{12}}x_{c_2}x_{c_{12}}^{-1},  
u_{c_2}\mapsto x_{c_{12}}x_{c_2}^{-1} $$
and the sequence of Nielsen transformations is given by

$\begin{cases}
x_{c_{11}}  \mapsto x_{c_{11}} & \mapsto x_{c_{11}} \\
x_{c_{12}}  \mapsto x_{c_{12}} & \mapsto x_{c_{12}}(x_{c_{12}}x_{c_2}^{-1})^{-1} \\
x_{c_2}  \mapsto x_{c_{12}}x_{c_2}^{-1} & \mapsto x_{c_{12}}x_{c_2}^{-1} \\
\end{cases}$
 
\end{example}

\section{Locally alternating augmented links}\label{section_alt}

Given a flat augmented link $L$, we construct a corresponding \textit{locally alternating augmented link}  $L_a$ as described in section \ref{section_thms} (see Figure \ref{L_a}).  Regarding the fibration of the oriented link $L_a$, we get a stronger statement. 

\begin{thm_alt}
Let $L_a$ be a locally alternating augmented link obtained from a flat augmented link $L$ with a connected diagram.  Then $L_a$ fibers. 
\end{thm_alt}

This theorem follows immediately from 
\begin{lem}
Let $L_a$ be a locally alternating augmented link obtained from a flat augmented link $L$ with a connected diagram. Then $L_a$ is obtained from $\pm 1$ Dehn filling  on the crossing circles of a flat augmented link $\tilde{L}$ such that $G_B(\tilde{L})$ is a tree.
\end{lem}

\begin{proof}
Just as for $L$, construct the standard Seifert surface for $L_a$, as well as the corresponding graph $G_B(L_a)$. Given $L_a$, suppose  $G_B(L_a)$ is not a tree. 

\begin{stepp}
Eliminate  nontrivial loops in $G_B(L_a)$ :
\end{stepp}  

Choose a loop $\gamma$ in $G_B(L_a)$ and an edge $e$ on this loop (a $B$-circle on $L_a$) connecting vertices $v_1, v_2$ corresponding to regions $C_1, C_2$. Figure \ref{remove_edge}  shows how  to obtain $L_a$ from $\pm 1$ filling on a $A$-circle of a new link $L_{a1}$.  The relationship between $G_B(L_a)$ and $G_B(L_{a_1})$ is  that $G_B(L_{a1})$ is obtained from $G_B(L_a)$ by adding a new vertex $v_3$ in the interior of $e$ and breaking the loop $\gamma$ at $v_3$.    

\begin{figure}[h]
\begin{center}
\includegraphics[scale=.11]{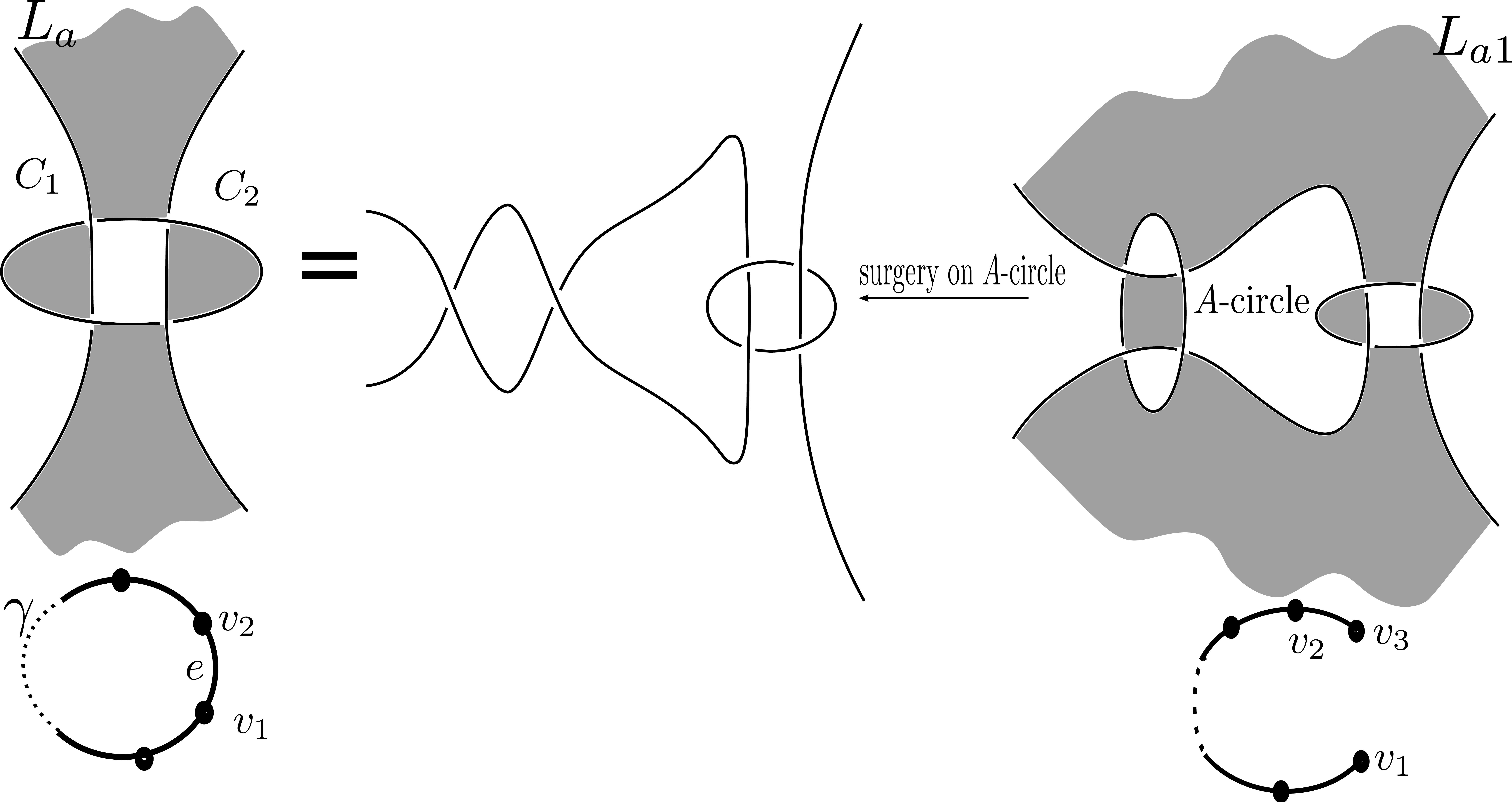}
\caption{Removing the loop $\gamma$ from $G_B(L_a)$.}
\label{remove_edge}
\end{center}
\end{figure}

\begin{stepp}
Join disconnected components of $G_B(L_a)$:
\end{stepp}  

Consider two vertices $v_1, v_2$ on $G_B(L_a)$ not in the same component but such that their corresponding $C$-regions $C_1,C_2$ are adjacent to a  $A$-circle. Figure \ref{join_vertices} shows how to obtain $L_a$ from $\pm 1$ filling on a $B$-circle of a new  link $L_{a2}$. The relationship here is  that $G_B(L_{a2})$ is obtained from $G_B(L_a)$ by joining $v_1,v_2$ by a new edge $e$. 

\begin{figure}[h]
\begin{center}
\includegraphics[scale=.15]{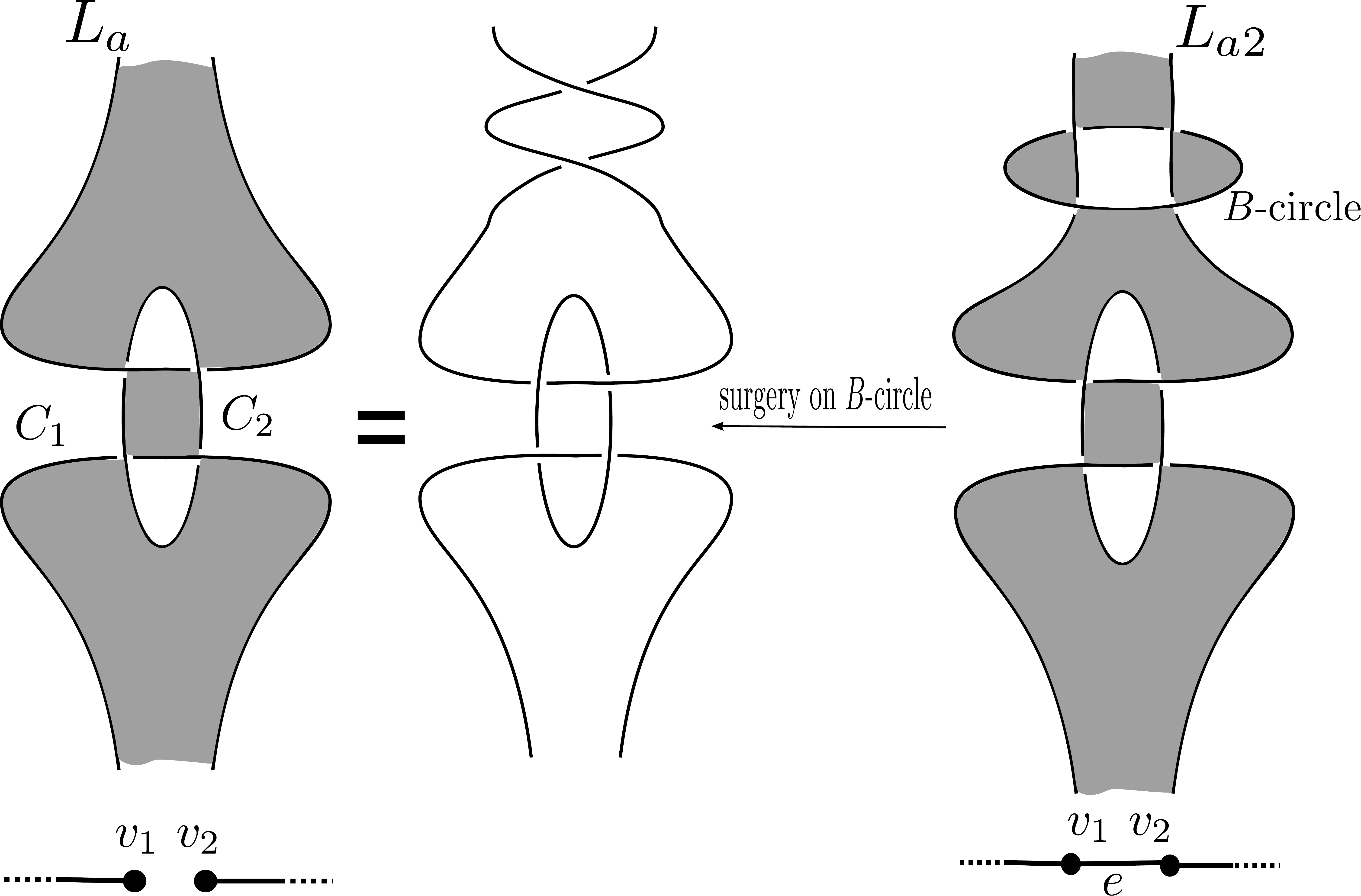}
\caption{Joining vertices $v_1$ and $v_2$.}
\label{join_vertices}
\end{center}
\end{figure}

\begin{stepp}
\end{stepp}
Repeat the steps above, as needed, until we obtain a link $\tilde{L}_0$ such that $G_B(\tilde{L}_0)$ is a tree.  Note  that some of its crossing circles are alternating and some are not.  $L_a$ is obtained from $\tilde{L}_0$ by a sequence of $\pm 1$ fillings on its flat crossing circles (these ones are the ones introduced in the steps above).

The conclusion now follows from the observation that   alternating crossing circles can be  obtained from  $\pm 1$ filling on a pair of flat crossing circles, as described in Figure \ref{surgery2}. The desired link $\tilde{L}$ is the link obtained from $\tilde{L}_0$ by replacing alternating crossing circles by a pair of non-alternating ones. Note that since $G_B(\tilde{L}_0)$ is a tree, $G_B(\tilde{L})$ is also a tree. 
\end{proof}

\begin{figure}[h]
\begin{center}
\includegraphics[scale=.17]{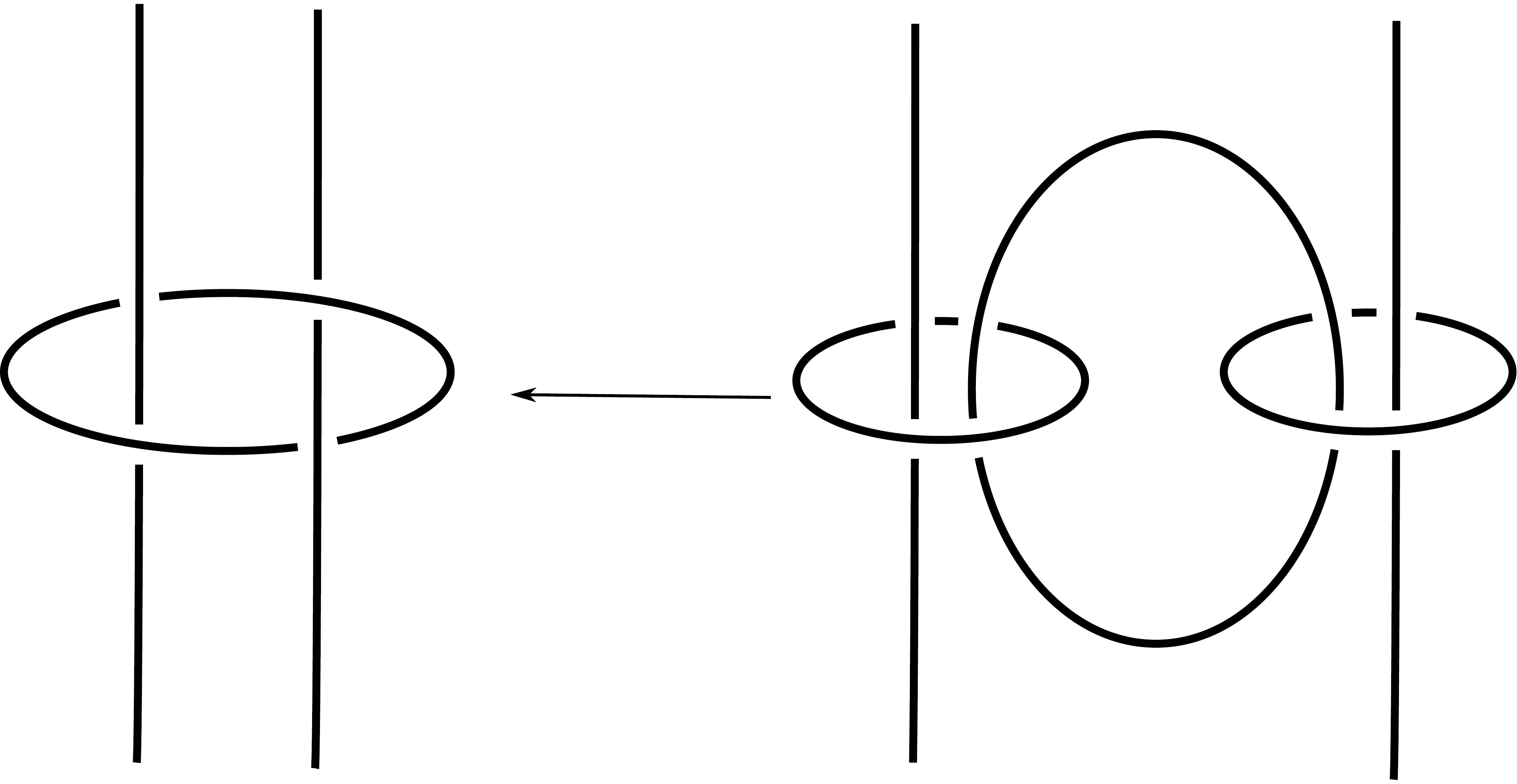}
\caption{Obtaining alternating crossing circles from fillings on non-alternating ones.}
\label{surgery2}
\end{center}
\end{figure}




\vspace{.4cm}
\noindent
\address{\textsc{Department of Mathematics,\\
 Universidade Federal do Cear\'a}}\\
\email{\textit{E-mail:} \texttt{dgirao@mat.ufc.br}}
\end{document}